\documentclass[12pt]{article}

\usepackage{t1enc}
\usepackage[latin1]{inputenc}
\usepackage[english]{babel}
\usepackage{amsmath,amsthm}
\usepackage{amsfonts}
\usepackage{latexsym}
\usepackage{graphicx}
\usepackage{txfonts}
\usepackage[natural]{xcolor}

\newtheorem{thm}{Theorem}[section]
\newtheorem{lem}[thm]{Lemma}
\newtheorem{cor}[thm]{Corollary}

\newtheorem{defn}[thm]{Definition}

\makeatletter

\newcommand{\Rmnum}[1]{\expandafter\@slowromancap\romannumeral #1@}
\makeatother

\begin{document}

\title{The edge chromatic number of outer-1-planar graphs\thanks{Supported by the National Natural Science Foundation of China (No.\,11301410, 11201440, 11101243), the Natural Science Basic Research Plan in Shaanxi Province of China (No.\,2013JQ1002), the Specialized
Research Fund for the Doctoral Program of Higher Education (No.\,20130203120021), and the Fundamental Research Funds for the Central Universities (No.\,K5051370003, K5051370021).}}
\author{Xin Zhang\thanks{Email address: xzhang@xidian.edu.cn.}\\
{\small School of Mathematics and Statistics, Xidian University, Xi'an 710071, P.\,R.\,China}}

\maketitle

\begin{abstract}\baselineskip  0.65cm
A graph is outer-1-planar if it can be drawn in the
plane so that all vertices are on the outer face and each edge is crossed
at most once. In this paper, we completely determine the edge chromatic number of outer 1-planar graphs.\\[.2em]
\end{abstract}

\baselineskip  0.65cm

\section{Introduction}

All graphs considered in this paper are simple and undirected. By $V(G),E(G),\Delta(G)$ and $\delta(G)$, we denote the set of vertices, the set of edges, the maximum degree and the minimum degree of a graph $G$, respectively. In any figure of this paper, the degree of a solid or hollow vertex is exactly or at least the number of edges that are incident with it, respectively. Moreover, solid vertices are distinct but two hollow vertices may be same unless we states.

A graph is \emph{outer-1-planar} if it can be drawn in the
plane so that all vertices are on the outer face and each edge is crossed at most once.  Outer-1-planar graphs were first introduced by Eggleton \cite{Eggleton} who called them \emph{outerplanar graphs with edge crossing number one}, and were also investigated under the notion of \emph{pseudo-outerplanar graphs} by Zhang, Liu and Wu \cite{ZPOPG}.
In fact, the notion of outer-1-planarity is a natural generation of the outer-planarity, and is also a combination of the 1-planarity and the outer-planarity.
From the definition of the outer-1-planarity, outer-1-planar graphs are a subfamily of planar graphs, which are one of the most studied areas in graph theory and an important class in graph drawing. It is now proved by Dehkordi and Eades \cite{DE} that every outer-1-planar graph has a right angle crossing drawing and by Auer \emph{et al.}\,\cite {ABBGHNR} that the recognition of outer-1-planarity can process in linear time. Outer-1-planar graphs are also used as a special graph family for verifying some interesting conjectures on graph colorings. For instance, it is proved that the list edge and the list total coloring conjectures hold for outer-1-planar graphs with maximum degree at least five \cite{TZ,LTC}, and the total coloring conjecture and the equitable $\Delta$-coloring conjectures hold for all outer-1-planar graphs \cite{TC,TZ}.

An \emph{edge $k$-coloring} of a graph $G$ is an assignment $f:E(G)\rightarrow \{1,2,\ldots,k\}$ so that $f(e_1)\neq f(e_2)$ whenever $e_1$ and $e_2$ are two adjacent edges. The minimum integer $k$ so that $G$ has an edge $k$-coloring, denoted by $\chi'(G)$, is the \emph{edge chromatic number} of $G$. The well-known Vizing's Theorem says that $\Delta(G)\leq \chi'(G)\leq \Delta(G)+1$ for every simple graph $G$. Therefore, to determine whether the edge chromatic number of a graph $G$ is $\Delta(G)$ or $\Delta(G)+1$ is interesting. However, the edge chromatic number problem is an NP-complete problem, and more badly, decide whether a given simple graph with maximum degree 3 has edge chromatic number 3 is also NP-complete \cite{Holyer}. As far as we know, the edge chromatic numbers of only few families of graphs have been fixed. For example, the edge chromatic numbers of 1-planar graphs with maximum degree at least 10 \cite{ZW}, planar graphs with maximum degree at least 7 \cite{SZ} and series-parallel graphs (thus also outerplanar graphs) with maximum degree at least 3  \cite{JMT} are the maximum degree.

The edge colorings of outer-1-planar graphs were first considered by Zhang, Liu and Wu \cite{ZPOPG}. They proved that the edge chromatic numbers of
outer-1-planar graphs with maximum degree at least 4 are the maximum degree and announced  that there are outer-1-planar graphs with maximum degree 3 and edge chromatic number 4. In this paper, we follow their work and determine the edge chromatic numbers of outer-1-planar graphs with maximum degree 3. Note that the edge chromatic numbers of graphs with maximum degree at most 2 can be easily fixed. Therefore, we completely determine the edge chromatic number of outer 1-planar graphs.

\section{The structures of outer-1-planar graphs with $\Delta=3$}

From now on, we assume that any outer-1-planar graph was drawn in the plane so that its outer-1-planarity is satisfied and the number of crossings is as few as possible, and this drawing is called an \emph{outer-1-plane graph}.
We follow the notations in \cite{ZPOPG}. Let $G$ be 2-connected outer-1-plane graph. Denote by $v_1, v_2, \ldots, v_{|G|}$ the vertices of $G$ that lie clockwise. Let $\mathcal{V}[v_i,v_j]=\{v_i, v_{i+1}, \ldots, v_j\}$ and $\mathcal{V}(v_i,v_j)=\mathcal{V}[v_i,v_j]\backslash \{v_i,v_j\}$, where the subscripts are taken modular $|G|$. Set $\mathcal{V}[v_i,v_i]=V(G)$ and $\mathcal{V}(v_i,v_i)=V(G)\setminus \{v_i\}$.
A vertex set $\mathcal{V}[v_i,v_j]$ with $i\neq j$ is \emph{non-edge} if $j=i+1$ and $v_iv_j\not\in E(G)$, is \emph{path} if $v_k v_{k+1}\in E(G)$ for all $i\leq k<j$, and is \emph{subpath} if $j>i+1$ and some edges in the form $v_kv_{k+1}$ with $i\leq k<j$ are missing. An edge $v_iv_j$ in $G$ is a \emph{chord} if $j-i\neq 1$ (mod $|G|$). By $\mathcal{C}[v_i,v_j]$, we denote the set of chords $xy$ with $x,y\in \mathcal{V}[v_i,v_j]$.

\begin{lem}\cite{ZPOPG}\label{path}
Let $v_i$ and $v_j$ be vertices of a 2-connected outer-1-plane graph $G$. If there is no crossed chords in $\mathcal{C}[v_i,v_j]$ and no edges between
$\mathcal{V}(v_i,v_j)$ and $\mathcal{V}(v_{j},v_{i})$, then $\mathcal{V}[v_i,v_j]$ is either non-edge or path.
\end{lem}


\begin{thm}\label{structure}
Every 2-connected outer-1-planar graph with maximum degree at most 3 contains one of the configurations $G_1,G_2,\dots,G_7$ and $H_t$ as in Figure \ref{o1p}. Moreover,\\
(a) if $G$ contains $G_2$ and $x\neq y$, then the graph derived from $G$ by deleting $u$ and identifying $v$ with $w$ is outer-1-planar;\\
(b) if $G$ contains $G_4$ and $x\neq y$, then the graph derived from $G$ by deleting $u_0,v_0,w$ and identifying $u_1$ with $v_1$ is outer-1-planar;\\
(c) if $G$ contains $G_8$ and $x\neq y$, then the graph derived from $G$ by deleting $u_0,u_1,v_0$ and identifying $u_2$ with $v_1$ is outer-1-planar;\\
(d) if $G$ contains $H_t$ and $x\neq y$, then the graph derived from $G$ by deleting $u_0,u_1,\ldots,u_t,v_0,v_1,\ldots,v_t$ and adding a new edge $xy$ (if it does not exist) is outer-1-planar.
\begin{figure}
  \begin{center}
  \includegraphics[width=4.5in]{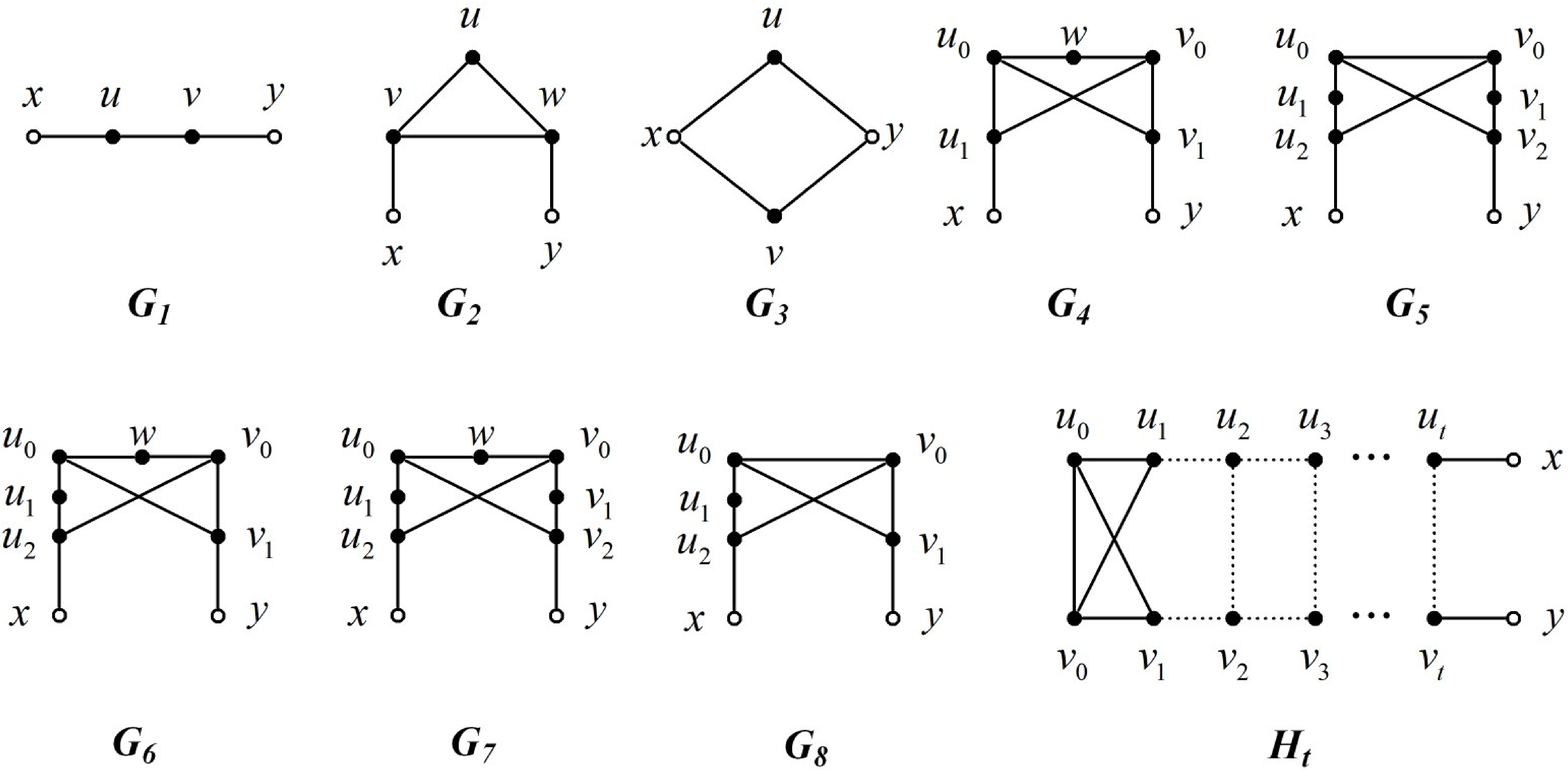}\\
  \end{center}

  \caption{Structures in outer-1-planar graph with maximum degree at most 3}\label{o1p}
\end{figure}
\end{thm}

\begin{proof}
We prove this result by contradiction. If there is no crossings in $G$, then $G$ is outerplanar and the results hold (cf.\,\cite{Wang}). Therefore we assume that crossings appear in $G$.
Let $v_iv_j$ and $v_lv_k$ be two mutually crossed chords in $G$ with $1\leq i<k<j<l$. Without loss of generality, assume that $i=1$ and there is no other pair of
mutually crossed chords in $\mathcal{C}[v_i,v_l]$. By Lemma \ref{path}, any of $\mathcal{V}[v_i,v_k],\mathcal{V}[v_k,v_j]$ and $\mathcal{V}[v_j,v_l]$ is either non-edge or path. Suppose that $k-i\geq 3$ and there is a chord $v_rv_s$ with $i\leq r<s\leq k$. Note that $\mathcal{V}[v_i,v_k]$ is path now. If $s-r\geq 3$, then the vertices $v_{r+1},\ldots,v_{s-1}$ are all of degree two, thus the configuration $G_1$ appears. If $s-r=2$, then $d(v_{r+1})=2$. If $d(v_r)=2$ or $d(v_s)=2$, then $G_1$ appears. If $d(v_r)=3$ and $d(v_s)=3$, then $G_2$ appears, and moreover, one can easily check that the condition (a) in the result we are proving holds. On the other hand, if $k-i\geq 3$ and there is no chords in $\mathcal{C}[v_i,v_k]$, then it is easy to see that $G_1$ appears. Therefore, we assume that $k-i\leq 2$, and similarly, assume that $j-k\leq 2$ and $l-j\leq 2$. If two of $\mathcal{V}[v_i,v_k],\mathcal{V}[v_k,v_j]$ and $\mathcal{V}[v_j,v_l]$ are non-edges, then we either find a 1-valent vertex in $G$ or have one another drawing of $G$ so that the number of crossing reduces one. Hence at least two of $\mathcal{V}[v_i,v_k],\mathcal{V}[v_k,v_j]$ and $\mathcal{V}[v_j,v_l]$ are paths. Suppose that $\mathcal{V}[v_i,v_k]$, $\mathcal{V}[v_k,v_j]$ are paths and $\mathcal{V}[v_j,v_l]$ is non-edge (the case when $\mathcal{V}[v_i,v_k]$ is non-edge and $\mathcal{V}[v_k,v_j]$, $\mathcal{V}[v_j,v_l]$ are paths is similar). If $j-k=k-i=1$, then $d(v_j)=2$ and $d(v_k)=3$, which implies either $G_1$ or $G_2$ occurs, and moreover, if $G_2$ appears, then (a) holds. If $j-k=1$ and $k-i=2$, then $d(v_{i+1})=d(v_j)=2$, which implies the appearance of $G_3$. If $j-k=2$, then $d(v_{j-1})=d(v_j)=2$ and $G_1$ appears. Suppose that $\mathcal{V}[v_i,v_k]$, $\mathcal{V}[v_j,v_l]$ are paths and $\mathcal{V}[v_k,v_j]$ is non-edge. If $k-i=l-j=1$, then $G_3$ occurs. If $k-i=2$ (the case when $l-j=2$ is similar), then $d(v_{k-1})=2$, which implies either $G_1$ or $G_2$ occurs, and moreover, one can check that (a) holds once $G_2$ appears. At last, we assume that $\mathcal{V}[v_i,v_k],\mathcal{V}[v_k,v_j]$ and $\mathcal{V}[v_j,v_l]$ are all paths.
If $j-k=2$ and $k-i=l-j=1$, then $G_4$ occurs, and moreover, (b) holds. If $k-i=2$ and $j-k=l-j=1$, or $l-j=2$ and $k-i=j-k=1$, then $G_8$ appear, and moreover, (c) holds. If $k-i=j-k=2$ and $l-j=1$, or $j-k=l-j=2$ and $k-i=1$, then $G_6$ appears. If $k-i=l-j=2$ and $j-k=1$, then $G_5$ appears. If $k-i=j-k=l-j=2$, then $G_7$ occurs. If $k-i=j-k=l-j=1$, then $d(v_k)=d(v_j)=3$. If $d(v_l)=2$, then $v_i$ is a cut vertex unless $G$ is $K_4-e$. Hence we assume $d(v_l)=3$ and $d(v_i)=3$ by symmetry. Let $v_r$ be a vertex of $G$ with $v_lv_r\in E(G)$ and $r>l$. Recall that we have assumed that $i=1$, thus $k=2,j=3$ and $l=4$.

\vspace{2mm} Case 1. $v_4v_r$ is a chord, i.e., $r\geq 6$.

\vspace{2mm} If $v_4v_r$ is non-crossed, then it is easy to see that $v_r$ disconnects the set $S=\{v_{5},\ldots,v_{r-1}\}\neq \emptyset$ and $V(G)\setminus S$, so $v_r$ is a cut-vertex, a contradiction. Hence we assume that $v_4v_r$ is crossed by another chord $v_xv_y$ with $x<r<y$.

\begin{figure}
 \begin{center}
  \includegraphics[width=4.5in]{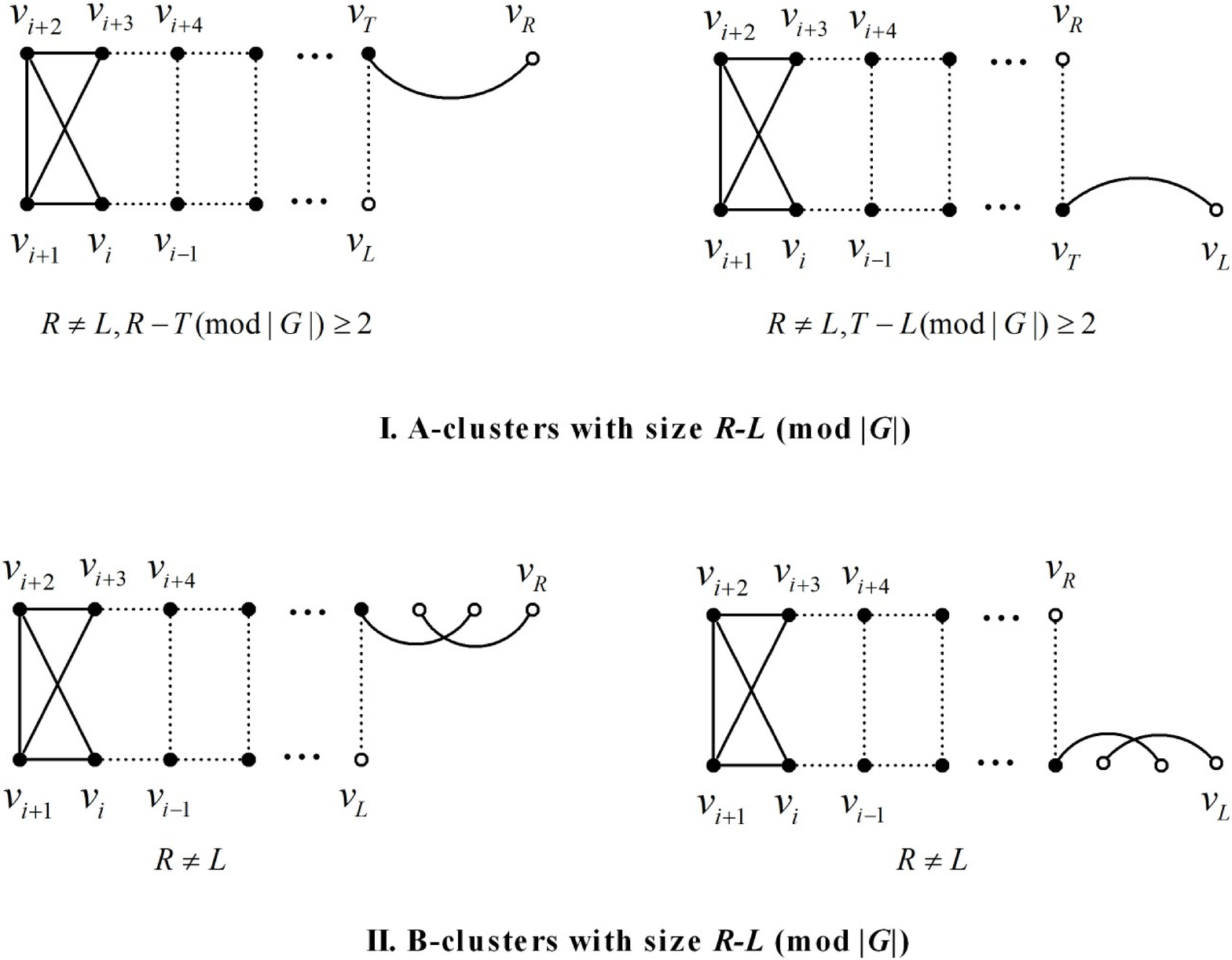}\\
 \end{center}

  \caption{Definitions of A-clusters and B-clusters}\label{cluster}
\end{figure}

\textbf{Notations:} The graphs that are isomorphic to any of the graphs in Figure \ref{cluster}-I and have the same drawings are called \emph{A-clusters} in $G$. The graphs that are isomorphic to any of the graphs in Figure \ref{cluster}-II and have the same drawings are called \emph{B-clusters} in $G$.
The \emph{size} of an A- or B-cluster is $R-L$ (mod $|G|$), where $R$ and $L$ are the subscripts of the far right vertex and the far left vertex (see in a clockwise direction from left to right) in the A- or B-cluster, respectively. If the size of an A- or B-cluster is smaller than another one A- or B-cluster, then we say the former A- or B-cluster is \emph{shorter} than the latter A- or B-cluster.
Note that every B-cluster contains a A-cluster.

For example, the graph induced by the edges $v_1v_2,v_1v_3,v_2v_3,v_2v_4,v_3v_4$ and $v_4v_r$ is an A-cluster with size $r-1$, the graphs induced by the edges $v_1v_2,v_1v_3,v_2v_3,v_2v_4,v_3v_4,v_4v_r$ and $v_xv_y$ is a B-cluster with size $y-1$.

Without loss of generality, assume that\\[.1em]
\textbf{(1)} \emph{there is no A-clusters with size less than $r-1$ in the graph induced by $\mathcal{V}[v_1,v_r]$,}\\
\textbf{(2)} \emph{there is no B-clusters with size less than $y-1$ in the graph induced by $\mathcal{V}[v_1,v_y]$.}\\[.1em]
Otherwise, we consider a shorter A- or B-cluster by similar arguments as below.

Suppose that there is a pair of crossed chords $v_{i'}v_{j'}$ and $v_{k'}v_{l'}$ with $4<i'<k'<j'<l'\leq x$. Similarly we can assume that $k'-i'=j'-k'=l'-j'=1$ and $d(v_{i'})=d(v_{l'})=3$. Note that $i'\neq 4$.
If $l'=x$, then let $v_{s'}$ be the third neighbor of $i'$. One can see that a copy of $H_1$ appears now, and moreover, $v_yv_{s'}\not\in E(G)$. It is easy to see that the graph obtained from $G$ by adding the new edge $v_yv_{s'}$ and removing $v_{l'}v_y$ is already outer-1-planar, hence (d) holds. Therefore we assume that $l'<x$.

If there is chord $v_{l'}v_{r'}$, then by (1), $4<r'\leq i'$. If $r'=i'$, then $G$ is disconnected, a contradiction, so we assume that $r'<i'$. Note that $v_{l'}v_{r'}$ is a crossed chord because otherwise $v_{r'}$ would be a cut vertex of $G$.
Since $d(v_{i'})=3$, there is an edge $v_{i'}v_{s'}$. If $s'>l'$, then we can redraw the graph by changing the order of $v_{i'},v_{k'},v_{j'}$ and $v_{l'}$ on the outer face to $v_{l'},v_{j'},v_{k'}$ and $v_{i'}$. After doing so, we avoid the crossing that generates by $v_{i'}v_{s'}$ crossing $v_{l'}v_{r'}$, which contradicts the fact that the drawing of $G$ minimizes the number of possible crossings. If $s'<i'-1$, then $v_{i'}v_{s'}$ is a chord and an A-cluster with size less than $r-1$ appears, a contradiction to (1). Hence $s'=i'-1$. Now one can see that a copy of $H_1$ appears in $G$. If $r'\neq s'$ and $v_{r'}v_{s'}\not\in E(G)$, then the graph derived from $G$ by adding the new edge $v_{r'}v_{s'}$ and removing $v_{l'}v_{r'}$ is already outer-1-planar, hence (d) satisfies.

If $v_{r'}v_{s'}\in E(G)$, then it is easy to see that $v_{l'}v_{r'}$ is crossed by a chord $v_{s'}v_{t'}$ with $4<t'<x$, and moreover, if $v_{r'}v_{s'}$ is a chord then it must be non-crossed. Note that the graph induced by $v_{i'}v_{j'},v_{i'}v_{k'},v_{j'}v_{l'},v_{k'}v_{j'},v_{k'}v_{l'}$ and $v_{l'}v_{r'}$ is an A-cluster. Without loss of generality, assume that\\[.1em]
\textbf{(3)} \emph{there is no A-clusters contained in the graph induced by $\mathcal{V}[v_{r'},v_{s'}]$,}\\[.1em]
otherwise we consider this A-cluster instead of the one we mentioned above.

Suppose that $v_{r'}v_{s'}$ is a chord. If there is no crossed chords in $\mathcal{C}[v_{r'},v_{s'}]$, then by Lemma \ref{path}, $\mathcal{V}[v_{r'},v_{s'}]$ is path, which implies the appearance of $G_1$ or $G_2$, and moreover, if $G_2$ appears then (a) holds. If there is a pair of crossed chords $v_{i''}v_{j''}$ and $v_{k''}v_{l''}$ with $r'<i''<k''<j''<l''<s'$, then we can assume that
$v_{i''}v_{k''},v_{k''}v_{j''},v_{j''}v_{l''}\in E(G)$, and furthermore, we have $i''\neq r'$, $l''\neq s'$ and $v_{i''-1}v_{i''},v_{l''}v_{l''+1}\in E(G)$ by (3). Now we see a copy of an $H_1$.
If $v_{i''-1}v_{l''+1}\not\in E(G)$, then adding an edge $v_{i''-1}v_{l''+1}$ to $G$ do not disturb its outer-1-planarity, hence (d) satisfies.
If $v_{i''-1}v_{l''+1}\in E(G)$, then by (3), we have $v_{l''+1}v_{l''+2},v_{i''-2}v_{i''-1}\in E(G)$ and thus a copy of $H_2$. We then discuss according whether $v_{i''-2}v_{l''+2}$ is an edge of $G$ or not and show that (d) satisfies. Here one can easily find that the next arguments are similar and iterative. Since the chord $v_{r'}v_{s'}$ is non-crossed and $\Delta(G)\leq 3$, we would finally find a copy of $H_k$ for some integer $k$ so that (d) satisfies and there is no way to construct a copy of $H_{k+1}$ based on this $H_k$. Hence, $v_{r'}v_{s'}$ cannot be a chord, which implies $r'=s'-1$.
If $4<t'<r'$, then $v_{t'}$ is a cut-vertex, a contradiction. If $t'>l'$, then redraw the graph by reserving the order of $v_{s'},v_{i'},v_{k'},v_{j'}$ and $v_{l'}$ on the boundary of the outer face. This would avoid the crossing generated by $v_{l'}v_{r'}$ crossing $v_{s'}v_{t'}$, contradicting the fact the the drawing of $G$ minimize the number of crossings.

Hence $v_{l'}v_{l'+1}\in E(G)$ and $v_{i'-1}v_{i'}\in E(G)$ by symmetry. Now we see a copy of an $H_1$.
If $v_{i'-1}v_{l'+1}\not\in E(G)$, then adding an edge $v_{i'-1}v_{l'+1}$ to $G$ do not disturb its outer-1-planarity, hence (d) satisfies.
If $v_{i'-1}v_{l'+1}\in E(G)$, then by similar arguments as above, we have $v_{l'+1}v_{l'+2},v_{i'-2}v_{i'-1}\in E(G)$ and thus a copy of $H_2$. We then discuss according whether $v_{i'-2}v_{l'+2}$ is an edge of $G$ or not and show that (d) satisfies. Here one can easily find that the next arguments are similar and iterative. Since there are finite vertices in $\mathcal{V}[v_4,v_x]$ and $v_4$ has no neighbors in $\mathcal{V}[v_4,v_x]$, we would finally find a copy of $H_k$ for some integer $k$ so that (d) satisfies and there is no way to construct a copy of $H_{k+1}$ based on this $H_k$. Therefore, there is no crossed chords in $\mathcal{C}[v_4,v_x]$, thus by Lemma \ref{path}, $\mathcal{V}[v_4,v_x]$ is either non-edge or path. Since $v_4$ has no neighbors in $\mathcal{V}[v_4,v_x]$, $\mathcal{V}[v_4,v_x]$ can only be a non-edge and thus $x=5$. By similar arguments as above, one can also show that there is no crossed chords in $\mathcal{C}[v_5,v_r]$
and thus $\mathcal{V}[v_5,v_r]$ is either non-edge or path. If $\mathcal{V}[v_5,v_r]$ is a non-edge, then $v_5$ is a 1-valent vertex, a contradiction. Hence $\mathcal{V}[v_5,v_r]$ is a path. If there is a chord in $\mathcal{C}[v_5,v_r]$, then it is easy to see that either $G_1$ or $G_2$ appear, and moreover, if $G_2$ occurs then (a) satisfies. Therefore, there is no chords in $\mathcal{C}[v_5,v_r]$. If $r\geq 7$, then $d(v_5)=d(v_6)=2$ and $G_1$ appears. Hence we assume that $r=6$ and $v_5v_6\in E(G)$. Note that $d(v_5)=2$.

Suppose that there is a pair of crossed chords $v_{i'}v_{j'}$ and $v_{k'}v_{l'}$ with $6\leq i'<k'<j'<l'\leq y$. Similarly as before, we can assume that $k'-i'=j'-k'=l'-j'=1$ and $d(v_{i'})=d(v_{l'})=3$, and moreover, assume that $l'\neq y$ (otherwise $v_6$ is a cut-vertex). If there is a chord $v_{l'}v_{r'}$ with $r'\neq k',i'$, then $v_{l'}v_{r'}$ is crossed because otherwise $v_{r'}$ would be a cut-vertex. If $r'>l'$, then by (2), $v_{l'}v_{r'}$ can only be crossed by a chord $v_{x'}v_{y'}$ with $l'<x'<t'$ and $6\leq y'\leq i'$. If $y'=i'$, then $v_{r'}$ is a cut-vertex, a contradiction, thus $6\leq y'< i'$. Since $d(v_{i'})=3$, there is an edge $v_{i'}v_{s'}$ with $s'<i'$. If $v_{i'}v_{s'}$ is a chord, then it is crossed by a chord $v_{a}v_b$ with $y'\leq b<s'<a<i'$, which implies a B-cluster with size less then $y-1$ in $G[\mathcal{V}[v_1,v_y]]$, a contradiction to (2). Hence $s'=i'-1$.
If $s'\neq y'$, or $s'=y'$ and $v_{y'}v_{r'}\not\in E(G)$, then $v_{s'}v_{r'}\not\in E(G)$. In this case a copy of $H_1$ appears, and moreover, the graph derived from $G$ by adding a new edge
$v_{s'}v_{r'}$ and removing the edge $v_{l'}v_{r'}$ is already outer-1-planar, thus (d) holds. If $s'=y'$ and $v_{y'}v_{r'}\in E(G)$, then $v_{r'}$ becomes a cut-vertex, a contradiction. Hence we assume that $6\leq r'\leq i'$.

If $r'=i'$, then $G$ has an isolate $K_4$, a contradiction, so suppose $6\leq r'<i'$. Since $d(v_{i'})=3$, there is an edge $v_{i'}v_{s'}$ with $s'\neq k',j'$. If $v_{i'}v_{s'}$ is a chord, then by similar argument as above, we shall assume that $s'>l'$. We redraw the graph by reversing the order of $v_{i'},v_{k'},v_{j'}$ and $v_{l'}$ on the boundary of the outer face. This operation reduces the number of crossings by one, a contradiction. Hence $s'=i'-1$, which implies that $r'\neq i'-1$, otherwise $v_{r'}$ is a cut-vertex. If $v_{i'-1}v_{r'}\not\in E(G)$, then the graph obtained from $G$ by adding an edge $v_{i'-1}v_{r'}$ and removing the edge $v_{r'}v_{l'}$ is already outer-1-planar, so (d) satisfies. If $v_{i'-1}v_{r'}\in E(G)$, then $v_{l'}v_{r'}$ can only be crossed by an edge that is incident with $v_{i'-1}$, say $v_{i'-1}v_{t'}$. If $t'<r'$, then it is easy to see that $v_{r'}$ is a cut-vertex, a contradiction. Hence we assume $t'>l'$. By similar arguments as the one after (3) we can claim that $r'=i'-2$ (here we shall assume, without loss of generality, that there is no A-clusters contained in the graph induced by $\mathcal{V}[v_{r'},v_{s'}]$). Therefore, we can reduce the number of crossings by one after redrawing the graph by reversing the order of $v_{i'-1},v_{i'},v_{k'},v_{j'}$ and $v_{l'}$ on the boundary of the outer face, a contradiction.

Therefore, $v_{l'}v_{l'+1}\in E(G)$ and $v_{i'-1}v_{i'}\in E(G)$ by symmetry.
Now we find a copy of a $H_1$.
If $v_{i'-1}v_{l'+1}\not\in E(G)$, then adding an edge $v_{i'-1}v_{l'+1}$ to $G$ do not disturb its outer-1-planarity, hence (d) holds.
If $v_{i'-1}v_{l'+1}\in E(G)$, then by similar arguments as above, we may have $v_{l'+1}v_{l'+2},v_{i'-2}v_{i'-1}\in E(G)$ and thus a copy of $H_2$. We then discuss according whether $v_{i'-2}v_{l'+2}$ is an edge of $G$ or not and show that (d) satisfies. Here one can again find that the next arguments are similar and iterative. Since there are finite vertices in $\mathcal{V}[v_6,v_y]$ and $v_5$ have no neighbors in $\mathcal{V}(v_6,v_y)$, we would finally find a copy of $H_k$ for some integer $k$ so that (d) satisfies and there is no way to construct a copy of $H_{k+1}$ based on this $H_k$. Therefore, there is no crossed chords in $\mathcal{C}[v_6,v_y]$, which implies by Lemma \ref{path} that $\mathcal{V}[v_6,v_y]$ is either a non-edge or a path. If it is a non-edge, then $d(v_6)=2$ and $G_1$ appears. If $\mathcal{V}[v_6,v_y]$ is a path, then $7\leq y\leq 8$ because otherwise $G_1$ occurs. If $y=8$, then $d(v_7)=2$ and thus $G_3$ appears. If $y=7$, then $d(v_6)=3$ and $G_2$ occurs, and moreover, (a) holds.

\vspace{2mm} Case 2. $v_4v_5,v_{|G|}v_1\in E(G)$.

\vspace{2mm} Note that a copy of $H_1$ appears now. If $v_{|G|}v_5\not\in E(G)$, then adding an edge $v_{|G|}v_5$ to $G$ do not disturb its outer-1-planarity, hence (d) holds. If $v_{|G|}v_5 \in E(G)$, then by similar arguments as in Case 1, we have $v_5v_6,v_{|G|-1}v_{|G|}\in E(G)$ and $H_2$ occurs. Obviously, the next arguments are
iterative and it is easy to see that (d) holds.
\end{proof}



\section{Edge coloring outer-1-planar graphs with $\Delta=3$}

We now investigate the edge colorings of outer-1-planar graphs with maximum degree 3. It is easy to see that the smallest (in terms of the order) outer-1-planar graph with $\Delta(G)=3$ and $\chi'(G)=4$ is the graph obtained from $K_5$ by removing two adjacent edges, say $K_5-2e$.

\begin{defn}\label{defn}
A graph $G$ is belong to the class $\mathcal{P}$, if it derives from $K_5-2e$ by a sequence of the following operations:
\begin{itemize}\vspace{-2mm}
  \item Remove a vertex $z$ of degree two and paste a copy of $G_2$, or $G_4$, or $G_8$ on the current graph by identifying $x$ and $y$ with $z_1$ and $z_2$, respectively, where $z_1$ and $z_2$ are the neighbors of $z$; \vspace{-2mm}
  \item Remove an edge $z_1z_2$ and paste a copy of $H_t$ for some integer $t$ on the current graph by identifying $x$ and $y$ with $z_1$ and $z_2$, respectively.
\end{itemize}
\end{defn}


The configurations $G_2,G_4,G_8$ and $H_t$ mentioned in above definition are the ones in Figure \ref{cluster}. One can easy to check that any graph $G\in \mathcal{P}$ has maximum degree 3 and minimum degree 2.

\begin{thm}\label{thmA}
If $G\in \mathcal{P}$, then $\chi'(G)=4$.
\end{thm}

\begin{proof}
Let $F$ be a graph in $\mathcal{P}$. 
If there is a vertex $z$ of degree two with neighbors $z_1$ and $z_2$ in $F$, then remove it and paste a copy of $G_2$ (or $G_4$, or $G_8$, respectively) on $H-z$ by identifying $x$ and $y$ with $z_1$ and $z_2$, respectively. Denote the current graph by $F_2$ (or $F_4$, or $F_8$, respectively). If $F_2$ (or $F_4$, or $F_8$, respectively) admits an edge 3-coloring $c$, then one can see that $c(vx)\neq c(wy)$ (or $c(u_1x)\neq c(v_1y)$, or $c(u_2x)\neq c(v_1y)$, respectively). Hence we can construct an edge 3-coloring of $F$ by restricting $c$ to $F-z=F_2-\{u,v,w\}=F_4-\{u_0,u_1,v_0,v_1,w\}=F_8-\{u_0,u_1,u_2,v_0,v_1\}$ and coloring $zz_1=zx$ and $zz_2=zy$ with $c(vx)$ and $c(wy)$ (or $c(u_1x)$ and $c(v_1y)$, or $c(u_2x)$ and $(v_1y)$, respectively). Therefore, if $\chi'(F)=4$ then $\chi'(F_2)=4$ (or $\chi'(F_4)=4$, or $\chi'(F_8)=4$, respectively).
Let $F_t$ be the graph derived from $F$ by applying the second operation in Definition \ref{defn} exactly once. If $F_t$ has an edge 3-coloring $c$, then one can check that $c(u_t x)=c(v_t y)$. Hence we can construct an edge 3-coloring of $F$ by restricting $c$ to $F-z_1z_2=F_t-\{u_0,\ldots,u_t,v_0,\ldots,v_t\}$ and coloring $z_1z_2=xy$ with $c(u_t x)$. Therefore, if $\chi'(F)=4$ then $\chi'(F_t)=4$.

As we can see now, any graph derived from a graph with edge chromatic number four by a sequence of the operations in Definition \ref{defn} still has edge chromatic number four. Since $\chi'(K_5-2e)=4$, $\chi'(G)=4$ for any $G\in \mathcal{P}$.
\end{proof}

\begin{thm}\label{thmB}
Let $\mathcal{O}_3$ be the family of outer-1-planar graphs with maximum degree 3. If $G\in \mathcal{O}_3\setminus \mathcal{P}$, then $\chi'(G)=3$.
\end{thm}

\begin{proof}
Let $G$ be a minimal counterexample to this statement. One can see that $G$ is 2-connected. By Theorem \ref{structure}, $G$ contains one of the configurations $G_1,G_2,\dots,G_7$ and $H_t$ as in Figure \ref{o1p}.

If $G$ contains $G_1$, then $G-uv$ has an edge 3-coloring $c$ by the minimality of $G$ and $c$ can be extended to $G$ by coloring $uv$ with a color different form $c(ux)$ and $c(uy)$.
If $G$ contains $G_3$, then $G-\{u,v\}$ has an edge 3-coloring $c$. If $d(x)=2$ or $d(y)=2$, then we come back to the case when
$G$ contains $G_1$. Let $x_1$ and $y_1$ be the third neighbor of $x$ and $y$, respectively. If $c(xx_1)=c(yy_1)=1$, then extend $c$ to an edge 3-coloring of $G$ by coloring $ux,vy$ with 2 and $vx,uy$ with 3. If $c(xx_1)=1$ and $c(yy_1)=2$, then extend $c$ to an edge 3-coloring of $G$ by coloring $vy$ with 1, $ux$ with 2 and $vx,uy$ with 3.
If $G$ contains $G_5$, then $G-\{u_0,u_1,v_0,v_1\}$ has an edge 3-coloring $c$. If $c(u_2x)=c(v_2y)=1$, then
extend $c$ to an edge 3-coloring of $G$ by coloring $u_0u_1,v_0v_1$ with 1, $u_1u_2,v_1v_2$ with 2 and $u_0v_2,v_0u_2$ with 3. If $c(u_2x)=1$ and $c(v_2y)=2$, then extend $c$ to an edge 3-coloring of $G$ by coloring $u_0v_2,v_0v_1$ with 1, $u_0v_0$ with 2 and $u_0u_1,v_1v_2,u_2v_0$ with 3.
If $G$ contains $G_6$, then $G-\{u_0,u_1,v_0,w\}$ has an edge 3-coloring $c$. If $c(u_2x)=c(v_1y)=1$, then
extend $c$ to an edge 3-coloring of $G$ by coloring $u_0u_1,wv_0$ with 1, $u_1u_2,u_0w,v_0v_1$ with 2 and $u_0v_1,u_2v_0$ with 3.
If $c(u_2x)=1$ and $c(v_1y)=2$, then extend $c$ to an edge 3-coloring of $G$ by coloring $u_0u_1,v_0v_1$ with 1, $u_0w,u_2v_0$ with 2 and $wv_0,u_1u_2,u_0v_1$ with 3.
If $G$ contains $G_7$, then $G-\{u_0,u_1,v_0,v_1,w\}$ has an edge 3-coloring $c$. If $c(u_2x)=c(v_2y)=1$, then
extend $c$ to an edge 3-coloring of $G$ by coloring $u_0u_1,v_0v_1$ with 1, $u_0w,u_2v_0,v_1v_2$ with 2 and $wv_0,u_0v_2,u_1u_2$ with 3.
If $c(u_2x)=$ and $c(v_2y)=2$, then extend $c$ to an edge 3-coloring of $G$ by coloring $wv_0,u_0v_2$ with 1, $u_0w,v_0v_1,u_1u_2$ with 2 and $u_0u_1,v_1v_2,u_2v_0$ with 3.

If $G$ contains $G_2$ and $x=y$, then $G=K_4-e$ since $G$ is 2-connected and $\chi'(G)=3$. If $G$ contains $G_2$ and $x\neq y$, then delete $u,vw$ and identify $v$ with $w$ as a common vertex $z$. Denote the resulted graph by $M_2$. If $\Delta(M_2)\leq 2$, then $\chi'(M_2)\leq 3$ by Vizing's theorem. If $\Delta(M_2)=3$, then by Theorem \ref{structure}(a) and Definition \ref{defn}, $M_2\in\mathcal{O}_3\setminus \mathcal{P}$, which implies that $\chi'(M_2)=3$ by the minimality of $G$. Let $c$ be an edge 3-coloring of $M_2$. Assume that $c(zx)=1$ and $c(zy)=2$. We construct an edge 3-coloring of $G$ by restricting $c$ to $G-\{u,v,w\}$ and coloring $vx,uw$ with 1, $uv,wy$ with 2 and $vw$ with 3.

If $G$ contains $G_4$ and $x=y$, then $G$ is the graph induced by the vertices of $G_4$ and one can check that $\chi'(G)=3$. If $G$ contains $G_4$ and $x\neq y$, then delete $u_0,v_0,w$ and identify $u_1$ with $v_1$ as a common vertex $z$. Denote the resulted graph by $M_4$. If $\Delta(M_4)\leq 2$, then $\chi'(M_4)\leq 3$. If $\Delta(M_4)=3$, then by Theorem \ref{structure}(b) and Definition \ref{defn}, $M_4\in\mathcal{O}_3\setminus \mathcal{P}$, which implies that $\chi'(M_4)=3$ by the minimality of $G$. Let $c$ be an edge 3-coloring of $M_4$. Assume that $c(zx)=1$ and $c(zy)=2$. We construct an edge 3-coloring of $G$ by restricting $c$ to $G-\{u_0,u_1,v_0,v_1,w\}$ and coloring $u_1x,u_0w,v_0v_1$ with 1, $u_0u_1,wv_0,v_1y$ with 2 and $u_0v_1,u_1v_0$ with 3.

If $G$ contains $G_8$ and $x=y$, then $G$ is the graph induced by the vertices of $G_8$ and one can check that $\chi'(G)=3$. If $G$ contains $G_8$ and $x\neq y$, then delete $u_0,v_1,v_0$ and identify $u_2$ with $v_1$ as a common vertex $z$. Denote the resulted graph by $M_8$. If $\Delta(M_8)\leq 2$, then $\chi'(M_8)\leq 3$. If $\Delta(M_8)=3$, then by Theorem \ref{structure}(c) and Definition \ref{defn}, $M_8\in\mathcal{O}_3\setminus \mathcal{P}$, which implies that $\chi'(M_8)=3$ by the minimality of $G$. Let $c$ be an edge 3-coloring of $M_8$. Assume that $c(zx)=1$ and $c(zy)=2$. We construct an edge 3-coloring of $G$ by restricting $c$ to $G-\{u_0,u_1,u_2,v_0,v_1\}$ and coloring $u_2x,u_0u_1,v_0v_1$ with 1, $u_1u_2,u_0v_0,v_1y$ with 2 and $u_0v_1,u_2v_0$ with 3.

If $G$ contains $H_t$ for some integer $t$, then $x\neq y$, because otherwise $G\in \mathcal{P}$. Moreover, we can assume, without loss of generality, that $xy\not\in E(G)$.
Delete all vertices of $H_t$ except $x$ and $y$ and connect $x$ with $y$ by an edge. By $M_t$ we denote the resulted graph. If $\Delta(M_t)\leq 2$, then $\chi'(M_t)\leq 3$. If $\Delta(M_t)=3$, then by Theorem \ref{structure}(d) and Definition \ref{defn}, $M_t\in\mathcal{O}_3\setminus \mathcal{P}$, which implies that $\chi'(M_t)=3$ by the minimality of $G$. Since the configuration $H_t$ is edge 3-colorable if and only if $u_tx$ and $v_ty$ receive same color, any edge 3-coloring $c$ of $M_t$ can be extended to an edge 3-coloring of $G$ by restricting $c$ to $G-xy$, coloring $u_tx,v_ty$ with $c(xy)$ and filling the colors on the remaining edges of the configuration $H_t$ properly.
\end{proof}

\section{Conclusions}

Combine Theorems \ref{thmA} and \ref{thmB} with Zhang, Liu and Wu' result \cite{ZPOPG} that every outer-1-planar graph with maximum degree $\Delta\geq 4$ has edge chromatic number $\Delta$, we have the following corollary, which completely determine the edge chromatic number of outer 1-planar graphs.

\begin{cor}
If $G$ is an outer-1-planar graph, then
$$\chi'(G)=\left\{
            \begin{array}{ll}
              \Delta(G), & \hbox{if $G\not\in \mathcal{P}$ and $G$ is not an odd cycle;} \\
              \Delta(G)+1, & \hbox{otherwise.}
            \end{array}
          \right.
$$
\end{cor}

On the other hand, since every graph $G\in \mathcal{P}$ has minimum degree 2, we have the following

\begin{cor}
If $G$ is a cubic outer-1-planar graph, then $\chi'(G)=\Delta(G)$.
\end{cor}

\noindent\textbf{Remark:} Not every graph in $\mathcal{P}$ is outer-1-planar graph. More precisely, a graph $G\in \mathcal{P}$ is outer-1-planar if and only if $G$ does not contain $K_4^+$ as a minor, where $K_4^+$ is the graph described in Figure \ref{K4+}, and furthermore, whether a graph $G\in \mathcal{P}$ is an outer-1-planar graph or not can be tested in linear time, see \cite{ABBGHNR}. On the other hand, whether an outer-1-planar graph with maximum degree 3 and minimum degree 2 belongs to $\mathcal{P}$ or not can also be decided in linear time by recognizing the configurations $G_2,G_4,G_8$ or $H_t$ in each step.

\begin{figure}

\begin{center}
\includegraphics[width=3cm,height=2.5cm]{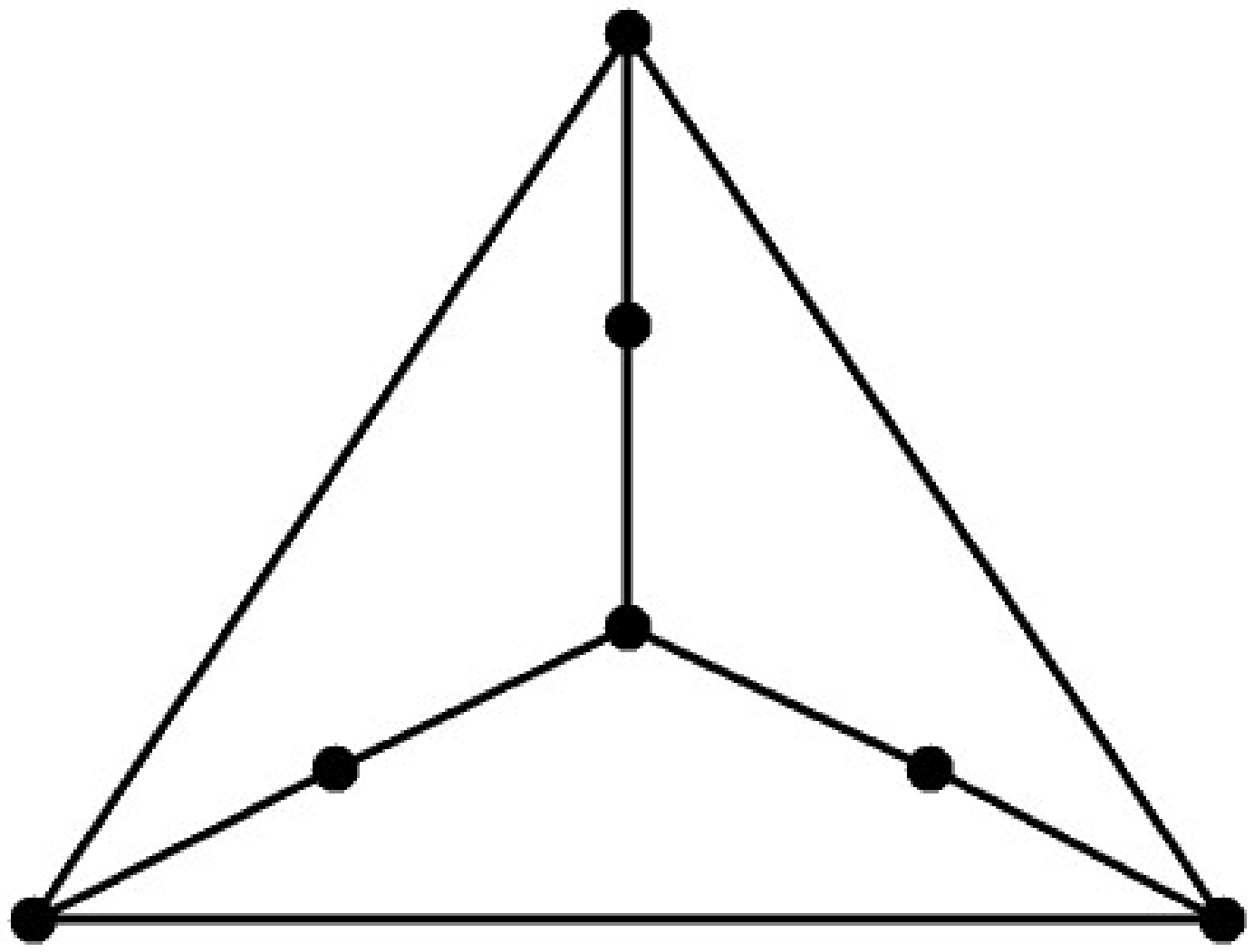}\\
\end{center}

  \caption{The graph $K_4^+$}\label{K4+}
\end{figure}


\begin{thebibliography}{10}\setlength{\itemsep}{-3pt}

\bibitem{ABBGHNR} C. Auer, C. Bachmaier, F. J. Brandenburg, \emph{et al.} Recognizing outer 1-planar graphs in linear time. LNCS 8242 (2013) 107--118.

\bibitem{Eggleton} R. B. Eggleton. Rectilinear drawings of graphs. Utilitas Math. 29 (1986) 149--172.

\bibitem{DE} H. R. Dehkordi, P. Eades. Every outer-1-plane graph has a right angle crossing drawing. Int. J. Comput. Geom. Appl. 22 (2012) 543--557.

\bibitem{Holyer} I. Holyer. The NP-completeness of edge-coloring. SIAM Journal on Computing 10(4) (1981) 718--720.

\bibitem{JMT} M. Juvan, B. Mohar, R. Thomas. List edge-coloring of series-parallel graphs. Electron. J. Combin. 6 (1999) R42.

\bibitem{SZ} D. P. Sanders, Y. Zhao. Planar graphs of maximum degree seven are class I. Journal of Combinatorial Theory, Series B 83(2) (2002) 348--360

\bibitem{TZ} J. Tian, X. Zhang. Pseudo-outerplanar graphs and chromatic conjectures. Ars Combinatoria 114 (2014) 353--361.

\bibitem{Wang} W. Wang, K. Zhang, $\Delta$-Matchings and edge-face chromatic numbers. Acta Math. Appl. Sinica 22 (1999) 236--242.

\bibitem{ZW} X. Zhang, J.-L. Wu. On edge colorings of 1-planar graphs. Information Processing Letters 111 (2011) 124--128.

\bibitem{ZPOPG} X. Zhang, G. Liu, J. L. Wu. Edge covering pseudo-outerplanar graphs with forests. Discrete Mathematics 312 (2012) 2788--2799.

\bibitem{TC} X. Zhang, G. Liu. Total coloring of pseudo-outerplanar graphs. arXiv:\,1108.5009v1 [math.CO]

\bibitem{LTC} X. Zhang. List total coloring of pseodo-outerplanar graphs. Discrete Mathematics 313 (2013) 2297--2306.


\end{thebibliography}
\end{document}